\documentclass[12pt,a4paper]{amsart}
\usepackage{amssymb,amsmath,pstcol}
\usepackage{hyperref}

\numberwithin{equation}{section}
\newtheorem{theorem}{Theorem}[section]
\newtheorem{lemma}[theorem]{Lemma}
\newtheorem{proposition}[theorem]{Proposition}
\newtheorem{corollary}[theorem]{Corollary}

\theoremstyle{definition}
\newtheorem{definition}[theorem]{Definition}
\newtheorem{example}[theorem]{Example}
\theoremstyle{remark}
\newtheorem{remark}[theorem]{Remark}

\newtheorem{problem}[theorem]{Problem}

\newcommand\Supp{\operatorname{Supp}}

\newcommand\Ass{\operatorname{Ass}}

\newcommand\Ann{\operatorname{Ann}}

\newcommand\height{\operatorname{height}}
\newcommand\grade{\operatorname{grade}}

\newcommand\Spec{\operatorname{Spec}}

\begin{document}
\title[$I$-Cohen Macaulay modules]{$I$-Cohen Macaulay modules}%
\author[W. Mahmood and M. Azam]{Waqas Mahmood and Maria Azam}
\address{\textbf{Waqas Mahmood}\\Department of Mathematics, Quaid-I-Azam University Islamabad, Pakistan.}%
\email{waqassms@gmail.com or wmahmood$@$qau.edu.pk}
\address{\textbf{Maria Azam}\\Department of Mathematics, Quaid-I-Azam University Islamabad, Pakistan.}%
\email{maria.azam555@gmail.com}

\maketitle
\begin{abstract}
A finitely generated module $M$ over a commutative Noetherian ring $R$ is called an $I$-Cohen Macaulay module, if
\[
\grade(I,M) + \dim(M/IM)= \dim(M),
\]
where $I$ is a proper ideal of $R$. The aim of this paper is to study the structure of this class of modules. It is discovered that $I$-Cohen Macaulay modules enjoy many interesting properties which are analogous to those of Cohen Macaulay modules. Also, various characterizations of $I$-Cohen Macaulay modules are presented here.
\end{abstract}

\section{Introduction}
The notions of Cohen Macaulay modules, as well as almost Cohen Macaulay modules are well established in commutative algebra. This paper is devoted to introduce and study the concept of $I$-Cohen Macaulay modules.

In the structure of $I$-Cohen Macaulay modules, we confine our attention entirely to arbitrary Noetherian ring which is not necessarily local. It is really interesting to establish the theory of $I$-Cohen Macaulay modules.  As might be expected, much of the motivation for our work comes from rapidly developing theory of commutative algebra.

A finitely generated module $M$ over a Noetherian ring $R$ is called an almost Cohen Macaulay module, if
\[
\grade(\mathfrak{p},M) = \grade(\mathfrak{p}R_\mathfrak{p}, M_\mathfrak{p}),
\]
for every  $\mathfrak{p}\in \Supp(M)$. Basically a flaw (If $S$ is commutative Noetherian ring and $M$ any finitely generated $S$-module, then $\grade(\mathfrak{p},M)=\grade(\mathfrak{p}R_\mathfrak{p},M_\mathfrak{p})\hspace{2mm}$ for every $\mathfrak{p}\in \Supp(M)$ ) \cite[15.C, page 97]{m} was corrected in \cite[Exercise 16.5]{m1}.

This led to the study of almost Cohen Macaulay rings and modules, first studied by Y. Han in 1998 (he referred it as $D$-ring) and M. Kang in 2001 \cite{k}. Some fascinating examples are given in \cite{k1} by Kang. Later, this class of modules is studied by several authors (see \cite{tang}, \cite{c}, \cite{c1}, \cite{mafi}, \cite{mafi1} and \cite{mafi2}.)

Ionescu discovered that how an almost Cohen-Macaulay module behaves when tensoring by a flat module. In this paper, this phenomena for $I$-Cohen Macaulay modules is discussed.

Many properties of $I$-Cohen Macaulay modules including the results about perseverance of $I$-Cohen Macaulayness in polynomial ring, formal power series ring and completion of the ring are given.

As for every proper ideal $I$ of $R$, a Cohen Macaulay module is $I$-Cohen Macaulay, it was mysterious (at least to the authors), that under what conditions an $I$-Cohen Macaulay module would be a Cohen Macaulay module? In this paper, a characterization regarding this issue is also given.

Throughout this paper, $R$ will be denoting a commutative Noetherian ring and $M$ a non-zero finitely generated $R$-module, unless otherwise specified.
Besides above mentioned characterizations some other results of this note are following.
\section{Main results}
A description of results regarding dimension and grade of modules is being given. For detailed study, see \cite{h}, \cite{m} and \cite{m1}.

Let $R$ be a local ring and $I$ a proper ideal of $R$. Then,
\[
r+\dim(M/(x_{1}, x_{2},...,  x_{r})M)\geq\dim(M).
\]
where $x_{1}, x_{2},...,  x_{r}\in I$.
 If $x_{1}, x_{2},...,  x_{r}\in I$ is a maximal $M$-regular sequence,
then,
\begin{equation}\label{e4}
\grade(I,M)+\dim(M/IM)\leq r+\dim(M/(x_{1}, x_{2},...,  x_{r})M)=\dim(M).
\end{equation}
Moreover, if $M=R$, then,
\[
\grade(I,R)+\dim(R/I)\leq \height(I)+\dim(R/I)\leq \dim R.
\]

\begin{definition}\label{d}
Let $R$ be a ring and $I$ a proper ideal of $R$. Then an $R$-module $M$ is said to be $I$-Cohen Macaulay, if the following equality holds:
\[
\grade(I,M) + \dim_R(M/IM)= \dim_R(M)
\]
\end{definition}

\begin{remark}\label{r} If $I$ is a proper ideal of a local ring $R$. Then $M$ is $I$-Cohen Macaulay for each of the following cases:
\begin{itemize}
\item [(1)] If $I$ is generated by an $M$-regular sequence, see Equation (\ref{e4}).

\item [(2)]  $M$ is Cohen Macaulay, see \cite[Theorem 2.1.2]{h}.
\end{itemize}
\end{remark}

Remark \ref{r} will not hold, provided that $R$ is not local. Moreover, it is shown that there is no relation between almost Cohen Macaulay module and $I$-Cohen Macaulay module, even if $I$ is a prime ideal.

\begin{example}\label{ex}
$(1)$ Let $R=A[x,y]$  where, $A=k[|U,V|]_{<U>}$  and $k$ a field. Note that $A$ is a DVR and $R$ a non-local Cohen Macaulay ring with $\dim R=3$. Consider the ideal $I=<Ux-1>$ of $R$. Then $\grade(I,R)=\height (I)=1$. There exists the following isomorphism
\[
R/I=\frac{A[y][x]}{<Ux-1>}\cong A_U[y], 
\]
Now $A$ is a DVR, so $A_U\cong k(|U,V|)$. Hence, $A_U[y]$ is a one dimensional P.I.D. Then $I$ is a prime ideal and $\dim R/I=\dim A_U[y]=1$. So,
\[
\grade(I,R)+\dim(R/I)=\height (I)+\dim R/I<\dim R.
\]
This proves that $R$ is not $I$-Cohen Macaulay.

$(2)$ Let $k$ be a field, $A=k[x_1,x_2,x_3,y_1,y_2,y_3]$  and $R=A/J$, where $J=<x_1,x_2,x_3>\cap <y_1,y_2,y_3>$. By \cite[Example, p-4]{k1}, $S$ is not an almost Cohen Macaulay ring. Let $I=<x_1,x_2,x_3>R$. Note that $ \Ass_R(R)=\{I, <y_1,y_2,y_3>\}$.
Hence, $\grade(I,R)=0$ and $\dim R/I=3.$ So,
\[
\grade(I,R)+\dim(R/I)= \dim R.
\]
\end{example}

\begin{remark}\label{r1}
$(1)$ If $R$ is $I$-Cohen Macaulay, then $\grade(I,R)= \height(I)$, but converse is not true, see Example \ref{ex}$(1)$,

$(2)$ If $R$ is local, $M$ an $I$-Cohen Macaulay module and $x_{1}, x _{2}, \cdots, x_{r}$ a maximal $M$-regular sequence in $I$. Form Equation (\ref{e4}), it follows that
\[
\dim(M/IM)= \dim(M/(x_{1}, x _{2}, \cdots x_{r})M).
\]
\end{remark}

Let $I$ be a non-zero proper ideal in $R$ and $M$ an $R$-module. In the next result, it will be shown that regular sequences are persevered under the following natural onto ring and module homomorphisms resp.:
\[
\phi_1:R\to R/\Ann_R(I)  \text{ and } \phi_2:M\to M/\Ann_R(I)M.
\]

\begin{proposition}\label{3}
With the previous notion, assume that $r_1,r_2,\dots, r_n\in I$ is an $M$-sequence such that $\phi_1({r_i})\notin \ker(\phi_1)$, for all $i=1,\dots, n$ (e.g. ${r_1},{r_2},\dots, {r_n}$ is $R$-regular). Then $\phi_1({r_1}),\phi_1({r_2}),\dots, \phi_1({r_n})\in \phi_1(I)$ is an $\phi_2(M)$-sequence.
\end{proposition}

\begin{proof}
Let $\overline{R}=\phi_1(R)$ and $\overline{M}=\phi_2(M)$.  For $n=1$, suppose that $r\in I$ is $M$-regular. Let $\overline{r}=r+\Ann_R(I)$ and $\overline{x}=x+\Ann_R(I)M\in \overline{M}$ such that
\[
\overline{r}\cdot\overline{x}=\overline{0} \text{ and } x\notin \Ann_R(I)M.
\]
It implies that $rx=\sum_{i=1}^s b_iy_i$ with $b_i\in \Ann_R(I)$ and $y_i\in M$ for all $i=1,\dots s$. Then
\[
r(Ix)=\sum_{i=1}^s (b_iI)y_i=0
\]
In particular, $r(rx)=0$. By regularity of $r$ on $M$, $x=0$ a contradiction. So, $\overline{r}\in \phi_1(I)$ is $\overline{M}$-regular.
Hence, the result follows by induction.
\end{proof}

\begin{corollary}\label{4}
Assume that $I$ is a non-zero proper ideal in a local ring $R$ and $\Ann_R(I)\subseteq I$. For, any $R$-module $M$, the following conditions are equivalent:
\begin{itemize}
  \item [(1)] $M$ is an $I$-Cohen Macaulay module.
  \item [(2)] $\phi_2(M)$ is $I\phi_2(M)$-Cohen Macaulay as a $\phi_1(R)$-module, $\grade(I,M)=\grade(\phi_1(I),\phi_2(M))$ and $\dim M=\dim \phi_2(M)$.
\end{itemize}
\end{corollary}

\begin{proof}
Let $M$ be an $I$-Cohen Macaulay module. Since
\[
\phi_2(M)/I\phi_2(M)\cong M/IM
\]
By Theorem \ref{3}, $\grade(I,M)\leq \grade(\phi_1(I),\phi_2(M))$. It implies that
\[
\begin{aligned}
\grade(\phi_1(I),\phi_2(M))
&\geq \dim(M) - \dim(M/I)\\
&\geq \dim(\phi_2(M))-\dim(\phi_2(M)/I\phi_2(M))\\
&\geq \grade(\phi_1(I),\phi_2(M)), \text{ see Equation (\ref{e4})}.
\end{aligned}
\]
This proves all the statements in $(2)$. The converse is obvious.
\end{proof}

If $f:(R,\mathfrak{m})\to (S,\mathfrak{n})$ is a local ring map of Noetherian local rings. Suppose $N$ is an $R$-flat finitely generated $S$-module. In
\cite[Theorem 2.1.7]{h} and  \cite[Proposition 2.2]{c}, the authors investigated behavior of Cohen macaulay and almost Cohen Macaulay modules upon tensoring $f$ with $N$. Next result discuses the behavior of $I$-Cohen Macaulay modules, when $f$ is tensoring with $N$.
\begin{theorem}\label{6}
With the above notion, assume that $I$ is any proper ideal of $R$. Let $M$ be an $R$-module and $M\otimes_R N\neq I(M\otimes_R N)$. Then following conditions are equivalent:
\begin{itemize}
  \item [(1)] $M$ is $I$-Cohen Macaulay and $\dim_S N/\mathfrak{m}N=0$.
  \item [(2)] $M\otimes_R N$ is $IS$-Cohen Macaulay, $\grade(I,M)=\grade(IS,M\otimes_R N)$ and
  \[
  \dim_R M/IM=\dim_S (M\otimes_R N)/I(M\otimes_R N).
  \]
  \item [(3)] $M\otimes_R N$ is $IS$-Cohen Macaulay, $\grade(I,M)=\grade(IS,M\otimes_R N)$ and
  \[
  \dim_R M=\dim_S (M\otimes_R N).
  \]
\end{itemize}
\end{theorem}

\begin{proof}
Only the implications $(1)\Rightarrow (2)$ and $(1)\Rightarrow(3)$ are proved here. The proof of $(2)$ implies $(1)$ is quit easy from Definition \ref{d} and \cite[A.11 Thoerem]{h}. Also, $(3)$ is equivalent to $(2)$ follows from the following isomorphism
\[
\frac{N/IN}{(\mathfrak{m}/I) (N/IN)}\cong N/\mathfrak{m}N.
\]
Assume that $M$ is $I$-Cohen Macaulay, Since, $\dim_S N/\mathfrak{m}N=0$, then
\begin{equation}\label{e7}
\grade(I,M)=\dim_S M\otimes_R N-\dim_R M/IM \text{ and }
\end{equation}
\[
\dim_{S/IS} (M/IM)\otimes_{R/I} (N/IN)=\dim_{R/I} (M/IM)+\dim_{S/IS} \frac{N/IN}{(\mathfrak{m}/I)(N/IN)},
\]
above equations hold because of  \cite[A.11 Thoerem]{h}. Since $\dim_{R/I} M/IM=\dim_{R} M/IM$ and
\[
(M/IM)\otimes_{R/I} (N/IN)\cong ( M\otimes_R N)/I( M\otimes_R N).
\]
So, last equation can be written as:
\[
\dim_{S}  ( M\otimes_R N)/I( M\otimes_R N)=\dim_{R} (M/IM)+\dim_{S} \frac{N/IN}{(\mathfrak{m}/I) (N/IN)}.
\]
Then, Equations (\ref{e4}) and (\ref{e7}) imply that
\[
\grade(I,M)
\geq \dim_S (M\otimes_R N)-\dim_S (M\otimes_R N)/I(M\otimes_R N)\geq \grade(IS, M\otimes_R N).
\]
By \cite[Proposition 1.1.2]{h}, $\grade(IS, M\otimes_R N)\geq \grade(I, M)$. This proves all the claims in $(2)$ and $(3)$.
\end{proof}

\begin{corollary}\label{8}
Let $f:R\to S$ be a ring map between Noetherian rings and $I$ any proper ideal of $R$. Suppose that $M$ and $N$ are finitely generated modules over $R$ and $S$ resp. If $N$ is flat over $R$, $I\subseteq\mathfrak{p}$ and $(M\otimes_R N)_\mathfrak{p}\neq I(M\otimes_R N)_\mathfrak{p}$, where $\mathfrak{p}=\mathfrak{q}\cap R$ and $\mathfrak{q}\in \Spec(S)$. Then the following conditions are equivalent:
\begin{itemize}
  \item [(1)] $M_\mathfrak{p}$ is $IR_\mathfrak{p}$-Cohen Macaulay and $\dim_{S_\mathfrak{q}}\frac{(M\otimes_R N)_\mathfrak{p}}{\mathfrak{p}(M\otimes_R N)_\mathfrak{p}}=0$.
  \item [(2)]  $(M\otimes_R N)_\mathfrak{p}$ is $IS_\mathfrak{q}$-Cohen Macaulay, $\grade(IR_\mathfrak{p},M_\mathfrak{p})=\grade(IS_\mathfrak{q},(M\otimes_R N)_\mathfrak{p})$ and
      \[
      \dim_{R_\mathfrak{p}} M_\mathfrak{p}/IM_\mathfrak{p}=\dim_{S_\mathfrak{q}}\frac{(M\otimes_R N)_\mathfrak{p}}{IS_\mathfrak{q}(M\otimes_R N)_\mathfrak{p}}.
      \]
  \item [(3)] $(M\otimes_R N)_\mathfrak{p}$ is $IS_\mathfrak{q}$-Cohen Macaulay, $\grade(IR_\mathfrak{p},M_\mathfrak{p})=\grade(IS_\mathfrak{q},(M\otimes_R N)_\mathfrak{p})$ and
      \[
      \dim_{R_\mathfrak{p}} M_\mathfrak{p}=\dim_{S_\mathfrak{q}} (M\otimes_R N)_\mathfrak{p}.
      \]
\end{itemize}
\end{corollary}

\begin{proof}
It is a consequence of Theorem \ref{6} in view of the fact that the ring map $R_\mathfrak{p}\to S_\mathfrak{q}$ is local.
\end{proof}

\begin{proposition}
 Let $I$ and $J$ be two proper ideals of $R$. Then
 \begin{itemize}
\item [(1)] If $M$ is $J$-Cohen Macaulay and $I\subseteq J$, then $M$ is $I$-Cohen Macaulay provided that $\grade(I,M)=\grade(J,M)$.
\item [(2)] If $M$ is $I$-Cohen Macaulay and $I\subseteq J$ , then $M$ is $J$-Cohen Macaulay provided that $\dim(M/JM)=\dim (M/IM)$.
\item [(3)] If $M$ is both $I$ and $J$-Cohen Macaulay, then $M$ is $I\cap J$-Cohen Macaulay.
 \end{itemize}
\end{proposition}

\begin{proof}
$(1)$ Since $\dim(\frac{M}{JM})\leq \dim(\frac{M}{IM})$ then,
\[
\grade(I,M)+ \dim(\frac{M}{IM})\leq \dim(M)= \grade(J,M) + \dim(\frac{M}{JM})\leq \grade(I,M)+ \dim(\frac{M}{IM})
\]
Hence, $M$ is $I$-Cohen Macaulay. With the similar arguments, $(2)$ can be proved. Also, $(3)$ follows from $(1)$ and \cite[Proposition 1.2.10]{h}.
\end{proof}

Cohen Macaulayness as well as almost Cohen Macaulayness are preserved when module is quotient by a regular sequence, see \cite[Theorem 2.1.3]{h} and \cite[Lemma 2.7]{k}. Same phenomena for $I$-Cohen Macaulay modules is given here.
\begin{theorem}\label{1}
Let $I$ be a proper ideal of a Noetherian local ring $R$, $M$ an $R$-module and $x_{1}, x _{2}, \cdots x_{r}$ an $M$-regular sequence in $I$. Then $M$ is $I$-Cohen Macaulay if and only if $\frac{M}{<x_{1}, x _{2}, \cdots x_{r}>M}$ is $\frac{I}{<x_{1}, x _{2}, \cdots x_{r}>}$-Cohen Macaulay over $\frac{R}{<x_{1}, x _{2}, \cdots x_{r}>}$.
\end{theorem}

\begin{proof}
Note that there is an isomorphism
\[
\frac{\acute{M}}{\acute{I}\acute{M}} \cong\frac{M}{IM},
\]
where $\acute{M} = \frac{M}{<x_{1},x _{2}, \cdots x_{r}>M}$ and $\acute{I} = \frac{I}{<x_{1},x _{2}, \cdots x_{r}>}$. So, the proof follows directly from Definition \ref{d}, Equation (\ref{e4}) and from the following equality between grades:
\[\grade(\frac{I}{<x_{1},x _{2}, \cdots x_{r}>},\frac{IM}{<x_{1},x _{2}, \cdots x_{r}>M})= \grade(I,M)- r,
\]
(see \cite[Proposition 1.2.10(d)]{h}).
\end{proof}

For any $\mathfrak{p} \in \Ass_{R}(M)$, a relation between dimension of a Cohen Macaulay module and dimension of $R/\mathfrak{p}$ is proved in \cite[Theorem 2.1.2]{h}. $i.e.$ a Cohen Macaulay module has no embedded primes. A similar relation is developed for $\mathfrak{p}$-Cohen Macaulay modules in the following result.

\begin{theorem}\label{6.6}
Let $M$ be a $\mathfrak{p}$-Cohen Macaulay module over a local ring $R$, where $\mathfrak{p}\in \Supp_R(M)$. Then
\[
\dim_R M=\dim_{R} R/\mathfrak{q}, \text{ for some } \mathfrak{q}\in \Ass_R(M) \text{ such that } \mathfrak{q}\subseteq \mathfrak{p}.
\]
\end{theorem}

\begin{proof}
If $\grade(\mathfrak{p}, M)=0$, then  $\mathfrak{q}\subseteq \mathfrak{p}$ for some $\mathfrak{q}\in \Ass_R(M)$. Since $\dim_R M/\mathfrak{p}M\leq\dim_R R/\mathfrak{p}$ and M is $\mathfrak{p}$-Cohen Macaulay, Then
\[
\text{ $\dim_R M=\dim_R R/\mathfrak{p}$ and $\mathfrak{p}=\mathfrak{q}$.}
\]
Now, if $\grade(\mathfrak{p}, M)>0$. Then, there exists an $M$-regular element $x\in \mathfrak{p}$. Let $\bar{M}=M/xM$, then, by Theorem \ref{1}, $\bar{M}$ is $\mathfrak{p}/<x>$-Cohen Macaulay. By induction on $\bar{M}$,
\[
\dim_{R} \bar{M}=\dim_{R/<x>} R/\mathfrak{q}_1=\dim_{R} {R}/\mathfrak{q}_1,
\]
\[
\text{ for some } \mathfrak{q}_1/<x>\in \Ass_{R/<x>}(\bar{M}) \text{ such that } \mathfrak{q}_1/<x>\subseteq \mathfrak{p}/<x>.
\]
By \cite[9.A Proposition]{m}, $\mathfrak{q}_1\in \Ass_R (\bar{M})$ such that $ \mathfrak{q}_1\subseteq \mathfrak{p}$. Then $\mathfrak{q}_1\in\Supp_R(\bar{M})$. In particular, $\mathfrak{q}_1\in\Supp_R(M)$ and $x\in \mathfrak{q}_1$. Then, there exists $\mathfrak{q}\in \Ass_{R}(M)$ such that $\mathfrak{q}\subsetneq \mathfrak{q}_1\subseteq \mathfrak{p}$. Hence
\[
\dim_R R/\mathfrak{q}\geq 1+ \dim_R R/\mathfrak{q}_1=1+\dim_{R} \bar{M}=\dim_R M.
\]
This proves that $\dim_R R/\mathfrak{q}=\dim_R M$, where $\mathfrak{q}\in \Ass_{R}(M)$ such that $\mathfrak{q}\subseteq \mathfrak{p}$.
\end{proof}

 \begin{lemma}\label{8.1}
With the same assumptions as in Theorem \ref{1}, assume that $S$ is a multiplicative closed subset of $R$ and $\mathfrak{p}\in \Supp_R(M)$ such that $\mathfrak{p}\cap S=\emptyset$. If $M$ is $\mathfrak{p}$-Cohen Macaulay, then
\begin{itemize}
  \item [(1)] $\grade(\mathfrak{p},M)=\grade(S^{-1}\mathfrak{p},S^{-1}M).$

\item [(2)] If $\dim_{R_\mathfrak{p}} M_\mathfrak{p}\geq 1$. Then, $\grade(\mathfrak{p}, M)\geq 1$.
\end{itemize}
\end{lemma}

\begin{proof}
$(1)$ If $\grade(\mathfrak{p}, M)=0$. Form the proof of Theorem \ref{6.6}, it follows that $\mathfrak{p}\in \Ass_R(M)$. Hence, $S^{-1}\mathfrak{p} \in \Ass_{S^{-1}R}(S^{-1}M)$ and $\grade(\mathfrak{p},M)=\grade(S^{-1}\mathfrak{p},S^{-1}M)=0$.

\par If $\grade(\mathfrak{p}, M)>0$. Then, there exists an $M$-regular element $x\in \mathfrak{p}$. Note that $x\in S^{-1}\mathfrak{p}$ is also $S^{-1}M$-regular. Let $\bar{M}=M/xM$. By induction on $\bar{M}$
\[
\grade(S^{-1}\mathfrak{p}, S^{-1}\bar{M})=\grade(\mathfrak{p}, \bar{M}).
\]
So, the result is proved by \cite[Proposition 1.2.10]{h}.

$(2)$ It is easy in view of arguments presented above.
\end{proof}

\begin{problem}
Perseverance of $I$-Cohen Macaulayness in module of fractions is still an open problem. However, in case of localization at a prime ideal is discussed in Theorem \ref{6.2}.
\end{problem}

By \cite[Theorem 2.1.7]{h}, localization at prime ideals of a Cohen Macaulay module is Cohen Macaulay. Same is true (with some conditions on ideal) for almost Cohen Macaulay modules, see \cite[Lemma 2.6]{h}. Further, system of parameters and maximal regular sequences coincide in Cohen Macaulay modules and their relation in almost Cohen Macaulay modules is described in \cite[Theorem 1.7]{k}. In case of $I$-Cohen Macaulay modules these results are developed in next Theorem.
\begin{theorem}\label{6.2}
With the previous notion, let $M$ be a $\mathfrak{p}$-Cohen Macaulay. Then the following statements hold:
\begin{itemize}
  \item [(1)] $M_\mathfrak{p}$ is Cohen Macaulay over $R_\mathfrak{p}$.
  \item [(2)] Ever maximal $M$-regular sequence in $\mathfrak{p}$ is a system of parameters of $M_\mathfrak{p}$. That is if $x_1,\dots,x_n\in \mathfrak{p}$ is a maximal regular sequence over $M$.  Then $\frac{x_1}{1},\dots,\frac{x_n}{1}\in \mathfrak{p}R_\mathfrak{p}$ is a system of parameters of $M_\mathfrak{p}$.
    \item [(3)] For any $\mathfrak{q} \in \Supp_R(M)$ such that $\mathfrak{q}\subseteq \mathfrak{p}$, $M_\mathfrak{q}$ is Cohen Macaulay over $R_\mathfrak{q}$.
\end{itemize}

\end{theorem}

\begin{proof}
Note that $(2)$ and $(3)$ are obvious in view of $(1)$ and Lemma \ref{8.1}. It is remain to prove the claim in $(1)$. If $\dim_R M=0$, then the result is true. So, let $\dim_R M\geq 1$, then by Lemma \ref{8.1}, $\grade(\mathfrak{p},M)\geq 1.$ So, there exists $x\in \mathfrak{p}$ which is $M$-regular. Also, it is $M_\mathfrak{p}$-regular in $\mathfrak{p}R_\mathfrak{p}$. Let $\bar{M}=M/xM$, then $\bar{M}$ is $\mathfrak{p}/x$-Cohen Macaulay, see Theorem \ref{1}. By induction, $\bar{M}_\mathfrak{p}$ is Cohen Macaulay over $\bar{R}_\mathfrak{p}$.
Hence $M_\mathfrak{p}/xM_\mathfrak{p}$ is Cohen Macaulay over $R_\mathfrak{p}/xR_\mathfrak{p}$. Now assertion follows from \cite[Proposition 19.3.3]{m}.
\end{proof}

\begin{remark}
By Example \ref{ex}$(1)$, the statements of Corollary \ref{6.2} are not true, if $R$ is not local.
\end{remark}

Since, the polynomial ring $R[x_1,\dots,x_n]$ and formal power series ring $R[|x_1,\dots,x_n|]$ inherit both Cohen Macaulayness and almost Cohen Macaulayness properties of $R$, see \cite[Theorem 2.1.9]{h} and \cite[Theorems 1.3 and 1.6]{k}. Same is true for $I$-Cohen Macaulay modules.
\begin{theorem}\label{5}
Let $I$ be a proper ideal of $R$. Then the following condition are equivalent:
\begin{itemize}
  \item [(1)] $R$ is $I$-Cohen Macaulay

   \item [(2)] $R[x_1,\dots,x_n]$ is $I[x_1,\dots,x_n]$-Cohen Macaulay and
  \[
      \grade(I[x_1,\dots,x_n],R[x_1,\dots,x_n])=\grade(I,R).
      \]
  \item [(3)] $R[|x_1,\dots,x_n|]$ is $I[|x_1,\dots,x_n|]$-Cohen Macaulay and
      \[
      \grade(I[|x_1,\dots,x_n|],R[|x_1,\dots,x_n|])=\grade(I,R).
      \]
\end{itemize}
\end{theorem}

\begin{proof}
It is enough to prove the result for $n=1$. Suppose that $(1)$ is true. Since $R \to R[x] $ is a flat ring map, then
\[
\begin{aligned}
\grade(I[x],R[x])
&\geq \grade(I,R)\\
&= \dim(R) - \dim(R/I)\\
&=\dim(R[x]) - \dim(R[x]/I[x])\\
&\geq \grade(I[x],R[x]), \text{ see Equation (\ref{e4})}.
\end{aligned}
\]
This proves $(2)$. The implication $(2)\Rightarrow (1)$ is straight forward. The proof of $(1)$ is equivalent to $(3)$ is so similar.
 \end{proof}

\begin{corollary}\label{5.1}
Let $I$ be any proper ideal of $R$ and $\mathfrak{p}\in \Spec(R)$ such that $\grade(I, R)=\height(I)$ and $\grade(\mathfrak{p}, R)=\height(\mathfrak{p})$. Then the following conditions hold:
\begin{itemize}
  \item [(1)] $R$ is $I$-Cohen Macaulay if and only if $R[|x_1,\dots,x_n|]$ is $I[|x_1,\dots,x_n|]$-Cohen Macaulay.

   \item [(2)] $R$ is $\mathfrak{p}$-Cohen Macaulay if and only if $R[x_1,\dots,x_n]$ is $\mathfrak{p}[x_1,\dots,x_n]$-Cohen Macaulay.
\end{itemize}
\end{corollary}

\begin{proof}
Since,$\phi: R\to R[|x|]$ is a faithfully flat ring map, see \cite[Exercise 10.12]{a}. By Theorem \ref{5}, the result is obvious in view of the following identities:
\[
\height(I)=\height(I[|x|]) \text{ and } \height(\mathfrak{p})=\height(\mathfrak{p}[x]),
\]
see \cite[(13.B)-Thoerem 19]{m} and \cite[Exercise 18]{g}
\end{proof}

Let $I$ be an ideal of $R$ contained in the Jacobson radical of $R$. Let $\hat{R}$ denote the $I$-adic completion of $R$. Then, \cite[Corollary 2.1.8]{h} and \cite[Corollary 2.4]{c}
state that $R$ is Cohen Macaulay ( resp. almost Cohen Macaulay) if and only if $\hat{R}$ is Cohen Macaulay ( resp. almost Cohen Macaulay).
\begin{theorem}
Assuming above notions, $R$ is $I$-Cohen Macaulay if and only if $\hat{R}$ is $I\hat{R}$-Cohen Macaulay and $\grade(I,R)=\grade(I\hat{R},\hat{R})$.
\end{theorem}

\begin{proof} The proof goes on the same lines to the proof of Theorem \ref{5} with following information:
\[
\text{$R \to \hat{R} $ is a faithfully flat ring map, see \cite[Theorem 56]{m}. }
\]
\[
\text{$\dim(\hat{R})= \dim(R)$ and $\hat{R}/I\hat{R}\cong R/I$, see \cite[p-175]{m}.}
\]
\end{proof}

\end{document}